\documentclass[letter]{amsart}
\usepackage{hyperref}
\usepackage{amsmath,amsfonts,amsthm, amssymb}
\usepackage{color}
\usepackage{graphicx}

\numberwithin{equation}{section}

\theoremstyle{plain}
\newtheorem{theorem}{Theorem}
\newtheorem{lemma}{Lemma}[section]
\newtheorem{proposition}{Proposition}[section]
\newtheorem{corollary}{Corollary}[section]

\theoremstyle{remark}
\newtheorem*{remark}{Remark}
\newtheorem{example}{Example}

\DeclareMathOperator{\tr}{tr}
\DeclareMathOperator{\id}{Id}

\newcommand{\R}{\mathbb{R}}
\newcommand{\dd}{\mathrm{d}}

\begin{document}

\title[On an elliptic system with symmetric potential]{On an elliptic system with symmetric potential possessing two global minima}
\author{Nicholas D.\ Alikakos}
\author{Giorgio Fusco}
\address{Department of Mathematics\\ University of Athens\\ Panepistemiopolis\\ 15784 Athens\\ Greece \and Institute for Applied and Computational Mathematics\\ Foundation of Research and Technology -- Hellas\\ 71110 Heraklion\\ Crete\\ Greece}
\email{\href{mailto:nalikako@math.uoa.gr}{\texttt{nalikako@math.uoa.gr}}} 
\address{Dipartimento di Matematica Pura ed Applicata\\ Universit\`a degli Studi dell'Aquila\\ Via Vetoio\\ 67010 Coppito\\ L'Aquila\\ Italy} \email{\href{mailto:fusco@univaq.it}{\texttt{fusco@univaq.it}}}

\begin{abstract}
We consider the system
\[
\Delta u-W_u(u)=0, \text{ for } u: \R^2 \to \R^2,\ W: \R^2 \to \R,
\]
where $W_u (u) := ( \partial W/ \partial u_1, \ldots, \partial W / \partial u_n )^{\top}$ is a smooth potential, symmetric with respect to the $u_1$, $u_2$ axes, possessing two global minima at $a^\pm=(\pm a,0)$ and two connections $e^\pm(x_1)$ connecting the minima. We prove that there exists an equivariant solution $u(x_1,x_2)$ satisfying
\begin{align*}
u(x_1,x_2) &\to a^\pm \text{ as } x_1\to \pm \infty, \\
u(x_1,x_2) &\to e^\pm(x_1) \text{ as } x_2 \to \pm \infty.
\end{align*}
The problem above was first studied by Alama, Bronsard, and Gui \cite{1}, under related hypotheses to the ones introduced in the present paper. At the expense of one extra symmetry assumption, we avoid their considerations with the normalized energy and strengthen their result. We also provide examples for $W$.
\end{abstract}

\maketitle

\section{Introduction}

The problem 
\begin{equation}
\label{problem}
\Delta u-W_u(u)=0, \text{ for } u: \R^2 \to \R^2,\ W: \R^2 \to \R,
\end{equation}
has variational structure and it is the Euler--Lagrange equation corresponding to the functional
\begin{equation}
J(u) = \int_{\R^2} \left\{ \frac{1}{2} |\nabla u|^2 + W(u) \right\} \dd x.
\label{action}
\end{equation}
An important feature of the problem, and also source of difficulty, is that the action $J$ is infinite for nonconstant solutions (\cite{22}).

Problem \eqref{problem} originates from geometric evolution and phase transitions. The relevant dynamical problem is the parabolic system
\[
\hat{u}_t = \varepsilon^2 \Delta \hat{u} - W_u(\hat{u}), \text{ for } \hat{u} : \R^2 \to \R^2.
\]
This is a gradient flow for the functional 
\[ 
\int_{\R^2} \left\{ \frac{1}{2} |\varepsilon\nabla u|^2 + W(u) \right\} \dd x,
\]
possessing diffused interfaces separating the minima of $W$.

\begin{figure}
\includegraphics[scale=1.05]{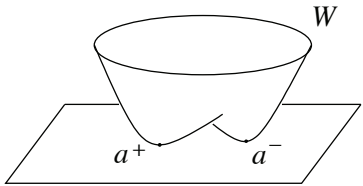}
\caption{The double-well potential $W$.}
\label{figure1}
\end{figure}

As was established in \cite{3}, there are multiple-well potentials $W$ for which the (ODE) connection problem between two phases admits more than one solution,
\begin{equation}
\ddot{U} - W_u(U) = 0, \text{ with } U(\pm\infty) = a^\pm. \label{connection-problem}
\end{equation}
To be specific, assume that there exist precisely two connections $e^\pm$, solutions to \eqref{connection-problem}. As a result, the diffused interfaces separating the phases $a^+$ and $a^-$ are generally made up of two types of `material', $e^+$ and $e^-$, one for each connection (see Figure 1 in \cite{3}). Simulations show that a wave is generated on the interface, which propagates and converts it into the type with lesser action 
\[
E(U) = \int_\R \left( \frac{1}{2} |\dot{U}|^2 + W(U) \right).
\]
The structure of the solution of $\hat{u}_t = \varepsilon^2 \Delta \hat{u} - W_u (\hat{u})$ close to the interface and near the junction is genuinely two dimensional (in $x_1$, $x_2$) and well-approximated by a suitably rescaled solution to the following traveling-wave problem
\begin{eqnarray} \left\{\begin{array}{l}
\label{traveling-wave-problem}
\Delta u - W_u(u) = -c \dfrac{\partial u}{\partial x_2},  \text{ for } u: \R^2 \to \R^2 \medskip\\
u(x_1, x_2) \to a^\pm, \text{ as } x_1\to \pm \infty, \medskip\\
u(x_1, x_2) \to e^\pm(x_1), \text{ as } x_2\to \pm \infty,
\end{array}\right.
\end{eqnarray}
with $\hat{u} (x_1,x_2,t)=u(x_1,x_2-ct)$, where $c$ is the speed of the wave, which can be shown to be proportional to $E(e^+) - E(e^-)$.

Problem \eqref{traveling-wave-problem} is rather difficult for $c \neq 0$ and is still open. The analogous ODE problem
\begin{eqnarray}
 \label{ode-problem}
\left\{\begin{array}{l}
\ddot{U} - W_u(U) = -c \dot{U} \medskip\\
U(\pm \infty) = a^\pm
\end{array}\right. 
\end{eqnarray}
was recently settled in \cite{4}.

In the case $E(e^+)=E(e^-)$, the wave becomes a standing wave, $c=0$, and \eqref{traveling-wave-problem} reduces to the problem studied in the present paper.

\medskip

Before stating our main result, we introduce the following hypotheses. \medskip

\noindent {\bf (H1)} ({\em Nondegeneracy}) The potential $W$ is $C^2$, $W:\R^2 \ to \R_+ \cup \{0\}$, and has exactly two nondegenerate global minima $a^\pm$, $a^\pm = (\pm a,0)$ with the properties $\partial^2 W(u) \geq c^2 \id$ and $|u - a^\pm| \leq r_0$, for $r_0>0$. \medskip

\noindent {\bf (H2)} ({\em Symmetry}) $W$ has dihedral symmetry, i.e., $W(gu) = W(u)$, for $g\in \mathcal{H}^2_2$, and the solution is equivariant, i.e., $u(gx)=gu(x)$, for all $g \in \mathcal{H}^2_2$. We assume that $W(u)\geq \max_{\partial C_0} W(u)$, for $u$ outside a certain bounded, $\mathcal{H}^2_2$-symmetric, convex set $C_0$. \medskip

\noindent {\bf (H3)} ({\em $Q$-monotonicity}) Let $\mathcal{D} := \{ (u_1, u_2) \mid u_1 > 0 \}$. We assume that there exists $Q: \bar{\mathcal{D}} \to \R_+ \cup\{0\}$, continuous, with the following properties:
\begin{enumerate}
\item[(i)] $Q$ is convex
\item[(ii)] $Q(u) > 0$ and $Q_u (u) \neq 0$ on $\mathcal{D} - \{a^\pm\}$
\item[(iii)] $Q(u+a^+) = |u|+H(u)$
\end{enumerate}
where $H = \bar{\mathcal{D}} \to \R$ is a smooth function that satisfies $H(0)=H_u(0)=0$, and
\[
W_u(u) \cdot Q_u(u) \geq 0, \text{ on } \mathcal{D} \setminus\{a^+\}.
\]

\noindent {\bf (H4)} The `scalar' trajectory $e_0$ which always exists by symmetry\footnote{The symmetry of $W$ assumed in (H2) implies that $\partial W / \partial u_2 (u_1,0) = 0$. Consequently, the solution of the scalar equation $e_{x_1 x_1} - \partial W / \partial u_1 (e,0) = 0$ with $e(\pm\infty)=\pm a$,
extends trivially to a solution of \eqref{problem} by setting $e_0(x_1) = (e(x_1),0)$. We normalize it by taking $e(0)=0$.} 
and as a curve lies on the $u_1$ axis and connects $a^+$, $a^-$, is assumed not to be a global minimum of the action
\[
E(U)=\int_\R \left\{ \frac{1}{2} |U_x|^2 + W(U) \right\} \dd x
\]
among the trajectories connecting $a^-$ and $a^+$. It follows by \cite{5} that there exists at least one pair of connecting trajectories $e_\pm$, which globally minimize the action in the class of trajectories that connect $a^\pm$ and with action strictly less than that of the scalar trajectory, $E(e_\pm) < E(e_0)$. \medskip

\noindent {\bf (H5)} Let $\mathcal{C}$ denote the set of connections between $a^+$ and $a^-$ and let $\mathcal{M}$ denote the set of globally minimizing connections. We assume that $\mathcal{C}$ is discrete and $\mathcal{C} \setminus \mathcal{M}$ is finite.

\medskip

Then, under the hypotheses above, we have the following

\begin{theorem}\label{theorem}
Under hypotheses {\em (H1)--(H5)} there exists a solution to
\[
\Delta u - W_u(u) = 0, \text{ for } u = \R^2 \to \R^2,
\]
which is $\mathcal{H}^2_2$-equivariant with the following properties:
\begin{enumerate}
\item[(i)] $u$ is a positive map, i.e., $u(\bar{\mathcal{D}}) \subset \bar{\mathcal{D}}$. \medskip
\item[(ii)] $|u(x) - a^+| < M e^{-c |x_1|}$, \text{ for } $x_1\ge0$ and $M$, $c$ constants. \medskip 
\item[(iii)] $\lim_{R \to \infty} \dfrac{1}{R} \int_{|x_2| < R} \left( \dfrac{1}{2} |\nabla u|^2 + W(u) \right) \dd x = E_{\min}$, where $E_{\min}$ is the value of $E$ on $\mathcal{M}$. \medskip 
\item[(iv)] The solution $u$ connects $a^\pm$ in the $x_1$-direction and a pair $e^\pm$ from $\mathcal{M}$ in the $x_2$-direction,
\begin{align*}
&\lim_{x_1 \to \pm\infty} u(x_1, x_2) = a^\pm, \medskip\\
&\lim_{x_2 \to \pm\infty} u(x_1, x_2) = e^\pm(x_1).
\end{align*}

\end{enumerate}
\end{theorem}

\begin{remark}
Due to the infinity of $J$ mentioned above, the solution is constructed as a limit of problems on strips of width $2R$, with $R \to \infty$. The main difficulty is showing that the limit
\[
u(x) = \lim_{R \to \infty} u_R (x)
\]
is nontrivial. The first enemy is $u \equiv 0$, but this is eliminated by the estimate in (ii). Another concern is that $u$ could coincide with one of the connections $e^+$ or $e^-$. This possibility is excluded by symmetry. This last point is considerably more involved in Alama, Bronsard, and Gui \cite{1}, since there only the symmetry with respect to the $u_2$-axis is assumed. Our method of proof in broad lines follows \cite{6}.
\end{remark}

Two related open problems are the following.

\subsubsection*{Multiplicity question.} If $\mathcal{M}$ has $k$ pairs of connections, then is it true that problem \eqref{problem} has $k$ distinct solutions $u_k$ with
\begin{align*}
&u_i(x_1,x_2) \to a^\pm, \text{ as } x_1 \to \pm\infty, \\
&u_i(x_1,x_2) \to e^\pm_i (x_2), \text{ as }x_2 \to \pm\infty,
\end{align*}
for $i = 1, \ldots, k$?

\subsubsection*{Diffeomorphism question} Is it true that the solution constructed in this paper is a global diffeomorphism, one-to-one, of $\R^2$ {\em onto} the region on the $u_1$--$u_2$ plane bounded by $e^+$ and $e^-$? Below we provide explicit examples of $W$. We note that the region bounded by the connections is convex for certain choices of the parameters.

\medskip

We conclude this introduction by giving examples of potentials $W$ satisfying the hypotheses (H1)--(H5).

\begin{example}\label{example1}
Consider the potential
\begin{eqnarray}\label{potential1}
W_1 (z) = \left| \frac{z^2-1}{z^2+\varepsilon^2} \right|^2, \text{ for } 0 <\varepsilon<\infty, 
\end{eqnarray}
where $z=u_1+iu_2$, $u=(u_1,u_2)$. The potential $W_1$ has two global minima at $a^\pm = (\pm1, 0)$ and obviously has the symmetry (H2). It has been shown in \cite{3} that there exist exactly three trajectories connecting $-1$ with $1, e^\varepsilon_+, e^\varepsilon_-$, and $e^\varepsilon_0$, with $e^\varepsilon_+, e^\varepsilon_-$ reflections of each other with respect to the $u_1$-axis and with $e^\varepsilon_0$ lying on the $u_1$-axis (see Figure \ref{figure2}a). Moreover, $E(e^\varepsilon_\pm) < E(e^\varepsilon_0)$ for $0<\varepsilon<\varepsilon^\ast = 0.4416 \ldots$ and $E(e^\varepsilon_\pm) > E(e^\varepsilon_0)$ for $\varepsilon > \varepsilon^\ast$. In more detail, the trajectories $e^\varepsilon_\pm$ are determined by the equation
\[
u_2 + \frac{1+\varepsilon^2}{4\varepsilon} \ln\left(\frac{(u_2-\varepsilon)^2 + u^2_1}{(u_2+\varepsilon)^2+u^2_1}\right) = 0
\]
and
\[
E(e^\varepsilon_0) = \frac{1}{\sqrt{2}}\left(\frac{1+\varepsilon^2}{\varepsilon}(\pi-\arctan\varepsilon)-\varepsilon\right), \ E(e^\varepsilon_\pm)=\frac{1}{\sqrt{2}}\left( 2 + \frac{2(1+\varepsilon^2)}{\varepsilon} \arctan\varepsilon \right).
\]

Modifying $W_1$ near the poles $\pm \varepsilon i$ allows us to produce a $C^\infty$ potential $\tilde{W}$ possessing the above trajectories. Clearly, the potential $\tilde{W}$ satisfies the hypotheses (H1), (H2), (H4), and (H5). For explaining the $Q$-monotonicity of $W$, condition (H3), we consider for the
moment the hypothesis
\[
W_u(u) \cdot (u-a^+) \geq 0, \text{ for } u \in D. \tag{H3*}
\]
Hypothesis (H3*) corresponds to the choice $Q(u)=|u-a^+|$ and states the monotonicity of $W$ along rays emanating from $a^+$.

For example, in the case when $W$ is a center at the origin, (H3$^\ast$) is never satisfied. On the other hand, the existence of a convex $Q$ which satisfies (H3) appears very plausible for centers and saddles but it would require proof. Our theorem produces an entire solution which appears to map the plane into the region bounded by the two symmetric connections.

\begin{figure}
\includegraphics[scale=0.825]{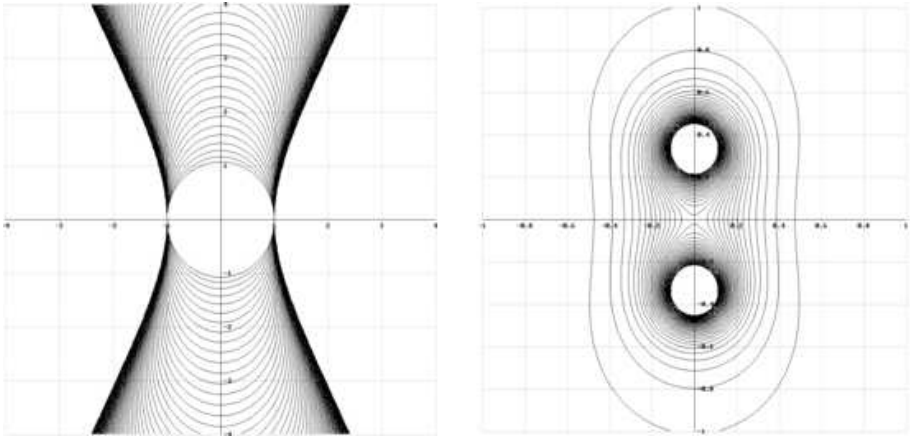}
\caption{The figure on the left shows a computation of the trajectories $e^\varepsilon_\pm$ for the potential $W_1$, for $0<\varepsilon<\infty$. We note that $e^\varepsilon_\pm$ tend to the unit circle, as
$\varepsilon \to 0$, while their envelope, as $\varepsilon \to \infty$, is given by $u^2_1 = u^2_2/3 + 1$. The disc-like boundary shown in the figure corresponds to $\varepsilon = \sqrt{3}/6 < \varepsilon^\ast = 0.4416 \ldots$ The region bounded by $e^\varepsilon_\pm$ ceases to be convex for $e = \sqrt{3}$. On the right we show the level sets of $W_1(z)$ for $\varepsilon = \sqrt{3}/6 < \varepsilon^\ast = 0.4416\ldots$ The existence of a $Q$ such that $Q_u \cdot W_u \geq 0$ in $D$ is geometrically plausible. (Numerical results due to G.\ Paschalides.)}
\label{figure2}
\end{figure}
\end{example}

\begin{example}
Consider the potential
\begin{equation}\label{potential2}
W_2 (z) = \left|\frac{z^2-1}{z^2+\varepsilon^2_1} \right|^2 \left| \frac{z^2-1}{z^2+\varepsilon^2_2} \right|^2, \text{ for } 0 \leq \varepsilon_1 \leq \varepsilon_2 < \infty, 
\end{equation}
where $z = u_1+iu_2$, $u=(u_1,u_2)$. The potential $W_2$ has global minima at $a^\pm=(\pm1,0)$ and obviously satisfies (H2). Applying the theory in \cite{3}, we get that for $\varepsilon_1>0$, there exist precisely five connecting orbits between
$a^+$ and $a^-$, which we denote by $e^1_\pm(\varepsilon_1,\varepsilon_2), e^2_\pm(\varepsilon_1, \varepsilon_2)$, and $e_0(\varepsilon_1,\varepsilon_2)$. We denote by $e_0$ the `scalar' connection mentioned in (H4) that lies on the $u_1$-axis, while the rest of the connections are symmetric in pairs with respect to the reflection $u_2 \mapsto -u_2$ (see Figure \ref{figure3}a) and are determined by the equation
\[
u_2 - \frac{(\varepsilon^2_1+1)^2}{4\varepsilon_1(\varepsilon_2-\varepsilon^2_1)}\ln \left(\frac{\varepsilon_1-u_2)^2+u^2_1}{(u_2+\varepsilon_1)^2+u^2_1}\right) + \frac{(\varepsilon^2_1+1)^2}{4\varepsilon_2(\varepsilon_2-\varepsilon^2_1)}\ln\left(\frac{(\varepsilon_2-u_2)^2+u^2_1}{(u_2+\varepsilon_2)^2+u^2_1}\right) = 0.
\]
In addition, by applying \cite{3}, the action of each orbit can be calculated explicitly.
\begin{align*}
E_0 := E(e_0) = \frac{1}{\sqrt{2}} &\bigg| 2-\frac{(\varepsilon^2_1+1)^2}{\varepsilon_1(\varepsilon^2_2-\varepsilon^2_1)}\arctan\varepsilon_1 + \frac{(\varepsilon^2_2+1)^2}{\varepsilon_2(\varepsilon^2_2-\varepsilon^2_1)}\arctan\varepsilon_2 \\ 
&- \frac{(\varepsilon^2_2+1)^2\pi}{2\varepsilon_2(\varepsilon^2_2-\varepsilon^2_1)} + \frac{(\varepsilon^2_1+1)^2\pi}{2\varepsilon_1(\varepsilon^2_2-\varepsilon^2_1)}\bigg|,
\end{align*}
\begin{align*}
E_{\textrm{I}} := E(e^1_\pm) = \frac{1}{\sqrt{2}} &\bigg| 2-\frac{(\varepsilon^2_2+1)^2}{\varepsilon_2(\varepsilon^2_2-\varepsilon^2_1)}\arctan\varepsilon_2 \\ 
&+ \frac{(\varepsilon^2_1+1)^2}{\varepsilon_1(\varepsilon^2_2-\varepsilon^2_1)}\arctan\varepsilon_1+\frac{(\varepsilon^2_2+1)^2\pi}{2\varepsilon_2(\varepsilon^2_2-\varepsilon^2_1)} \bigg|,
\end{align*}
\[
E_{\textrm{II}} := E(e^2_\pm) = \frac{1}{\sqrt{2}} \left|2 + \frac{(\varepsilon^2_2+1)^2}{\varepsilon_2(\varepsilon^2_2-\varepsilon^2_1)}\arctan\varepsilon_2-\frac{(\varepsilon^2_1+1)^2}{\varepsilon_1(\varepsilon^2_2-\varepsilon^2_1)}\arctan\varepsilon_1\right|.
\]

\begin{figure}
\includegraphics[scale=0.825]{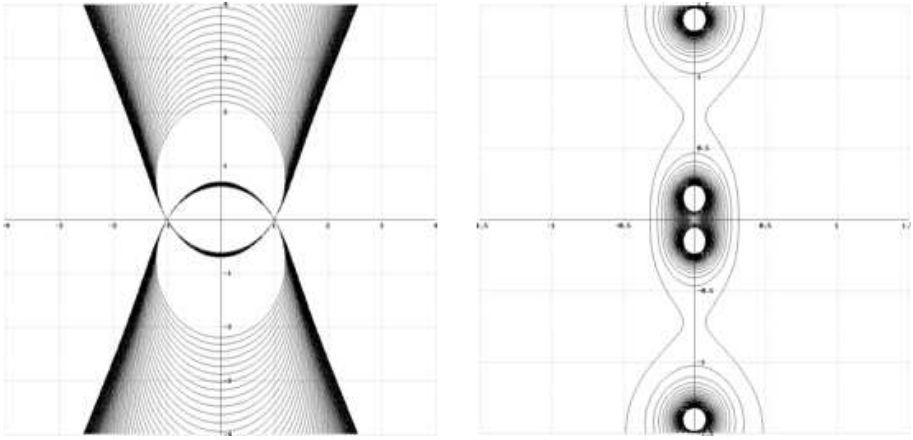}
\caption{The figure on the left shows a computation of the trajectories $e^1_\pm(\varepsilon_1, \varepsilon_2),e^2_\pm(\varepsilon_1,\varepsilon_2)$ for the potential $W_2$, for $\varepsilon_1$ fixed and equal to $\textstyle{\varepsilon^\ast_1 = \sqrt{\frac{\sqrt{6}-1}{2}} - \frac{\sqrt{6}-1}{2}}$ and $\varepsilon_2=(\sigma (\varepsilon^\ast_1), +\infty)$. It can be seen that the inner region approaches a limiting shape as $\varepsilon_2 \to 0$. On the right are the level sets of $W_2(z)$ for $\varepsilon_1 = \varepsilon^\ast_1$ and $\varepsilon_2 = \sigma(\varepsilon^\ast_1)$. (Numerical results due to G.\ Paschalides.)}
\label{figure3}
\end{figure}

We observe that $e^2_0 \cup e^2_+$ form the boundary of a region which increases unboundedly as $\varepsilon_2 \to \infty$ but approaches a limiting region as $\varepsilon_2 \to 0$, always enclosing all poles $(0,\pm\varepsilon_k i)$, for $k=1,\, 2$. On the other hand, $e^1_+ \cup e^1_-$ form the boundary of an interior region which contains only one pair of poles and approaches limiting regions as $\varepsilon_2 \to 0$ and $\varepsilon_2 \to \infty$.

We note that
\[
E_0 (\varepsilon_1, \varepsilon_2) \to \left\{
\begin{array}{l}
\infty, \text{ as } \varepsilon_1 \to 0, \\
\text{finite limits, as } \varepsilon_1,\, \varepsilon_2 \to \varepsilon^\ast \neq 0\\ 
\text{and also as } \varepsilon_1 \to\infty \text{ or } \varepsilon_2 \to \infty.
\end{array}\right.
\]
\[
E_{\textrm{I}} (\varepsilon_1, \varepsilon_2) \to \left\{ 
\begin{array}{l}
\infty, \text{ as } \varepsilon_2 - \varepsilon_1 \to 0, \\
\text{finite limits, as } \varepsilon_1 \to 0,\, \varepsilon_2 \to \infty.
\end{array}\right.
\]
\[
E_{\mathrm{II}} (\varepsilon_1, \varepsilon_2) \to \left\{
\begin{array}{l}
\text{finite limits, as } \varepsilon_2 \to 0, \\
\infty, \text{ as } \varepsilon_2 \to \infty.
\end{array}\right.
\]
From the previous relations, we get that
\begin{align*}
&E_{\mathrm{I}} > E_{\mathrm{II}}, \text{ as } \varepsilon_2 \to 0 \text{ (with $\varepsilon_1$ held
constant)}, \\
&E_{\mathrm{I}} < E_{\mathrm{II}}, \text{ as } \varepsilon_2 \to \infty \text{ (with $\varepsilon_1$ held
constant)}.
\end{align*}
It then follows easily that there exists a continuous function
$\varepsilon_1 \mapsto \sigma^\ast (\varepsilon_1)$ and $\varepsilon^\ast_1 > 0$ such that
\begin{eqnarray}\label{relations}
E_{\mathrm{II}} (\varepsilon_1,\sigma^\ast(\varepsilon_1)) = E_{\mathrm{I}} (\varepsilon_1, \sigma^\ast(\varepsilon_1)) < E_0(\varepsilon_1, \sigma^\ast(\varepsilon_1)),
\end{eqnarray}
for $0 \leq \varepsilon_1 < \varepsilon^\ast_1$. Thus, $\# m = 4$ and $\# C = 5$.

The theorem applies to $C^\infty$ modifications of $W_2$ with $\varepsilon_2 = \sigma^\ast(\varepsilon_1)$, $0 \leq \varepsilon_1 < \varepsilon^\ast_1$, and produces an $\mathcal{H}^2_2$ equivariant solution, apparently not unique, which has the property that
\[
u (x_1, x_2) \to e^i_+(x_1) \text{ and } u(x_1,-x_2) \to e^i_-(x_1), \text{ as } x_2 \to +\infty,
\]
for $i = 1,\, 2$. We expect that for the example at hand it should be possible to prove that there exist two distinct solutions $u^i$ satisfying
\[
\lim_{x_2 \to \pm\infty} u^i (x_1, x_2) = e^i_\pm(x_1), \text{ for } i=1,2,
\]
each mapping the plane $\R^2$ diffeomorphically to the region bounded by the corresponding connections.
\end{example}

\section{The constrained problem (H1)}\label{constrained-sec}
Let 
\[
\Omega_{R,\mu} = \left\{ (x_1, x_2) \mid |x_1| < \mu R,\, |x_2| < R \right\}
\]
 and
\[
C^+_{R,\mu,\eta} = \left\{ (x_1, x_2) \in \Omega_{R,\mu} \mid \eta R \leq x_1 \leq \mu R \right\},
\]
where $R \in [1,\infty)$, $\mu \in [1,+\infty]$, $1/2 < \eta < \mu$, and 
\[
C_{R, \mu, \eta} = \{ (x_1, x_2) \in \Omega_{R,\mu} \mid -\mu R \leq x_1 \leq -\eta R\}.
\]
Finally, the domain $\Omega_{R,\infty}$, for $\mu = \infty$ is the strip $|x_2|<R$. Consider the equivariant Sobolev space
\[
W^{1,2}_E (\Omega_{R, \mu}) = \{u: \Omega_{R,\mu} \to \R^2  \mid u \in W^{1,2}(\Omega_{R,\mu}),\, u\ \mathcal{H}^2_2 \text{-equivariant}\}.
\]
We consider, for $r<r_0$ fixed, the set
\begin{equation}\label{u-crm}
U^c_{R, \mu} := \{ u \in W^{1,2}_E (\Omega_{R,\mu}) \mid |u(x)-\alpha^\pm| \leq r,\, \text{ a.e.\ } x \in C^\pm_{R,\mu,\eta}\}
\end{equation}
and the functional
\[
J_{R,\mu}(u) = \int_{\Omega_{R,\mu}} \left\{ \frac{1}{2} |\nabla u|^2 + W(u) \right\} \dd x.
\]


\begin{proposition}\label{constrained}
Let $1 \leq R < \infty$, $1 \leq \mu \leq \infty$, $1/2 < \eta \leq \mu$, and $r<r_0$ fixed, where $r_0$ as in {\em (H1)}. Then, the problem
\begin{equation}\label{min-problem}
\min_{U^c_{R,\mu}} \int_{\Omega_{R,\mu}} \left\{ \frac{1}{2} |\nabla u|^2 + W(u) \right\} \dd x := \min_{U^c_{R,\mu}} J_{R,\mu} 
\end{equation}
has a solution $u_{R,\mu} \in W^{1,2}_E (\Omega_{R,\mu})$ for $\mu < \infty$ and $u_{R,\infty} \in (W^{1,2}_E)_{\textrm{loc}} (\Omega_{R,\infty})$
\end{proposition}

\begin{proof}
For $\mu < \infty$, we fix $R$ and $\mu$ and define the affine function $u_{\text{aff}} : \Omega_{R,\mu} \to \R^2$, such that
\begin{equation}\label{u-aff}
u_{\text{aff}}(x) := 
\left\{\begin{array}{ll}
a^-, & \text{for } x_1 \in [-\mu R, -1], \medskip\\
\dfrac{1-x_1}{2} a^- + \dfrac{1+x_1}{2} a^+, & \text{for } x_1 \in [-1,1], \medskip\\
a^+, & \text{for } x_1 \in [1,\mu R].
\end{array}\right.
\end{equation}
The function $u_{\text{aff}}$ belongs to $U^c_{R,\mu}$ for every $R\geq 1$, $\mu \geq \eta$, and satisfies the estimate
\begin{equation}\label{u-aff-estimate}
J_{R,\mu} (u_{\text{aff}}) < CR.
\end{equation}
Since $W \geq 0$, it follows that $0 \leq \inf_{U^c_{R,\mu}} J_{R,\mu} < J_R(u_{\text{aff}}) < CR$, where,
without loss of generality, we assumed the middle inequality to be strict. Let $\{u_n\}$ be a minimizing sequence of $J_{R,\mu}$, that is, $J_{R,\mu}(u_n) \to \inf_{U^c_{R,\mu}} J_{R,\mu}$. For the sequence $\{u_n\}$ we have the following estimates
\begin{eqnarray}\label{un-estimates}
\left\{\begin{array}{ll}
(i) & \displaystyle{\int_{\Omega_{R,\mu}} \dfrac{1}{2} |\nabla_{u_n}|^2\, \dd x < J_{R,\mu} (u_{\text{aff}}) < CR,} \medskip\\
(ii) & \displaystyle{\int_{\Omega_{R,\mu}} |u_n|^2\, \dd x < C(R,\mu),}
\end{array} \right.
\end{eqnarray}
where in \eqref{un-estimates}(ii) $C(R,\mu)$ denotes a constant depending on $R$, $\mu$. Then, there exists a subsequence, by weak compactness, which we still denote by $\{u_n\}$, such that
\[
u_n \rightharpoonup u, \text{ weakly in } W^{1,2}_E(\Omega_{R,\mu}).
\]
By lower semi-continuity in $L^2_E(\Omega_{R,\mu})$, it follows that
\begin{equation}\label{lim-inf-1}
\liminf_{n\to\infty} \int_{\Omega_{R,\mu}} |\nabla u_n|^2\, \dd x \geq \int_{\Omega_{R,\mu}} |\nabla u|^2\, \dd x 
\end{equation}
and by the compactness of the embedding $W^{1,2}_E(\Omega_{R,\mu}) \subset\subset L^2_E(\Omega_{R,\mu})$ and Fatou's lemma, we have
\begin{equation}\label{lim-inf-2}
\liminf_{n\to\infty} \int_{\Omega_{R,\mu}} W(u_n)\, \dd x \geq \int_{\Omega_{R,\mu}} W(u)\, \dd x.
\end{equation}

For handling the $\mu=\infty$ case, consider a family of rectangles $[-m,m]\times[-R,R]$, $m=1,3,\ldots$ First, construct a sequence minimizing $J$ over $W^{1,2}(\Omega_{R,\infty})$ functions restricted to the $m=1$ rectangle. Next, consider a subsequence of the previous sequence restricted to the $m=2$ rectangle and minimizing $J$ over $W^{1,2}(\Omega_{R,\infty})$ functions restricted to the
$m=2$ rectangle, and so on (via (\eqref{lim-inf-1}-type of estimates).

By a diagonal argument one obtains a subsequence $\{u_m\}$ which converges weakly in $W^{1,2}_{\textrm{loc}}(\Omega_{R,\infty})$ to some $u$. Utilizing the compactness of the embedding $W^{1,2}_{\textrm{loc}} \hookrightarrow L^2_{\textrm{loc}}$, we may assume that $u_m \to u$ a.e.\ (at the expense of taking a further subsequence). Now $\{\inf J_{R,\mu}\}$ is a decreasing sequence in $\mu$ and clearly $\inf J_{R,\mu} \geq \inf J_{R,\infty}$. That actually $\inf J_{R,\mu} \to \inf J_{R,\infty}$ as $\mu \to\infty$ follows from the fact that $C^\infty$ functions with compact support in $\R \times [-R,R]$ are dense in $W^{1,2}(\Omega_{R,\infty})$. Therefore,
\begin{align*}
\inf J_{R,\infty} &= \liminf_{\mu \to \infty} \int_{\Omega_{R,\mu}} \left( \frac{1}{2} |\nabla u_m|^2 + W(u_m) \right) \dd x \medskip\\
&\geq \int_{\Omega_{R,\infty}} \left( \frac{1}{2} |\nabla u|^2 + W(u) \right) \dd x,
\end{align*}
where in the last inequality we utilize Fatou's lemma. The proof is complete.
\end{proof}

\section{The positivity property}
Let $V$ be a real Euclidean vector space, and let $O(V)$ stand for the orthogonal group. For every finite subgroup $G$ of $O(V)$ a {\em fundamental region} is defined as a set $F$ with the following properties.
\begin{enumerate}
\item $F$ is open in $V$,
\item $F \cap TF = \varnothing$ if $\id \neq T \in G$,
\item $V = \cup \{(\overline{TF}) \mid T \in G\}$,
\end{enumerate}
where with the overbar we denote the closure of the set. The fundamental region $F$ can be chosen to be convex, actually a simplex (see \cite{10}). More generally, if $X$ is a subset of $V$, invariant under $G$, then a subset $D$ is a {\em fundamental domain} if it is of the form
\[
D=X\cap F.
\]
If $G=\mathcal{H}^2_2$, a fundamental region is $F=\{(u_1,u_2) \mid u_1 \geq 0, u_2\geq 0\}$. For $X=\Omega_{R,\mu}$, we take as a fundamental domain the set $\Omega^1_{R,\mu} = \Omega_{R,\mu} \cap F$.


\begin{proposition}[H2]\label{proposition-fundamental}
Let $u_{R,\mu}$, for $R,\mu \in [1,\infty]$, be the minimizing function of the constrained problem
\eqref{min-problem}. Then, there exists $u^\ast_{R,\mu} \in U^c_{R,\mu}$ with the properties
\begin{equation}\label{u*-properties}
\left\{\begin{array}{l}
J(u^\ast_{R,\mu}) \leq J(u_{R,\mu}), \medskip\\
u^\ast_{R,\mu}(\Omega^1_{R,\mu}) \subset \overline{F}.
\end{array}\right.
\end{equation}
\end{proposition}

\begin{proof}
Set
\begin{equation}\label{lambda}
\Lambda u := 
\left\{\begin{array}{ll}
u, & u\in F \medskip\\
T^{-1}_1u, & u\in T_1(F) \medskip\\
(T_2T_1)^{-1}u, & u\in T_2T_1(F)=S(F) \medskip\\
T^{-1}_2u, & u\in T_2(F).
\end{array}\right.
\end{equation}
Clearly, $\Lambda$ maps $\R^2$ into $F$. Also, it can be checked that
\begin{equation}\label{lambda-contraction}
|\Lambda(u_A) - \Lambda(u_B)| \leq |u_A - u_B|,
\end{equation}
where $|\cdot|$ is the Euclidean norm.

Next, we define the operator
\begin{equation}\label{l-operator}
(Lu)(x) := \Lambda u(x), \text{ for } x\in\Omega^1_{R,\mu},
\end{equation}
and extend by equivariance on $\Omega_{R,\mu}$. We will show that
\begin{equation} \label{l-preserves}
L : U^c_{R,\mu} \to U^c_{R,\mu},
\end{equation}
which means that $L$ preserves Sobolev equivariance and the constraint.

We begin by verifying that $L$ preserves Sobolev equivariance. By standard approximation arguments, the only source of difficulty is the possible loss of continuity along the symmetry lines where the gluing in the definition of $L$ takes place. We check two cases and leave the rest to the reader.

\begin{figure}
\includegraphics[scale=0.22]{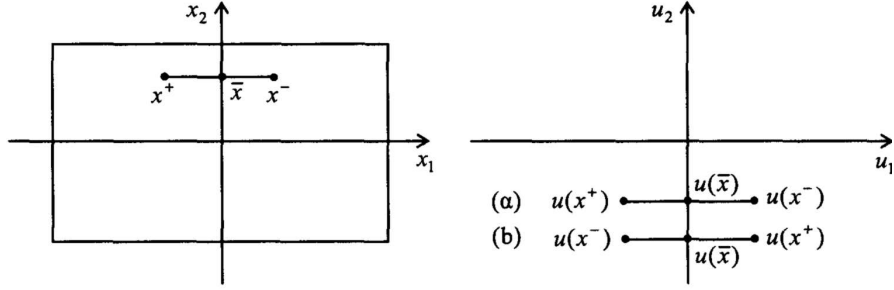}
\caption{The points $x^+$, $\bar{x}$, $x^-$, and the corresponding $u(x^-)$, $u(\bar{x})$, $u(x^+)$.}
\label{fig:x-points}
\end{figure}

We consider $x^+$, $\bar{x}$, $x^-$ as in Figure \ref{fig:x-points}a with $T_1 x^+ = x^-$ and $|x^+ - x^-|$ small, and $T_1 \bar{x} = \bar{x}$. We would like to show that $|(Lu)(x^+) - (Lu)(x^-)|$ is small for
$|u(x^+) - u(x^-)|$ small. By equivariance, $T_1(u(\bar{x})) = u(T_1\bar{x}) = u(\bar{x})$ and therefore, $u(\bar{x})$ lies on the $u_2$-axis. We assume that $u(x^-)$, $u(\bar{x})$, $u(x^+)$ are as in
Figure \ref{fig:x-points}. Then,
\[
\begin{array}{ll}
& Lu(x^-)=\Lambda u(x^-)=T_2u(x^-), \medskip\\
& Lu(x^+)=T_1\Lambda u(T^{-1}_1x^+)=T_1\Lambda u(x^-)=T_1T_2u(x^-)=T_2T_1u(x^-)=T_2u(x^+), \medskip\\
& Lu(x^-)=\Lambda u(x^-)=T_1T_2u(x^-)=T_2T_1u(x^-)=T_2u(x^+), \medskip\\
& Lu(x^+)=T_1\Lambda u(T^{-1}_1x^+)=T_1\Lambda u(x^-)=T_1T_1T_2u(x^-)=T_2u(x^-),
\end{array}
\]
consequently, continuity is verified in these cases. The verification of the constraint is straightforward. Finally, we define
\[
u^\ast_{R,\mu} := Lu_{R,\mu}
\]
and verify that $u^\ast_{R,\mu}$ does not increase the functional $J$. Indeed,
\[
W((Lu)(x)) = W(g\Lambda u(g^{-1}x)) = W(\Lambda u(g^{-1}x)) = W(u(g^{-1}x))
\]
and consequently, the term $W$ of the functional $J$ does not change since $T_i$ is an isometry. On the other hand, the term $\int_{\Omega_{R,\mu}} |\nabla u|^2\, \dd x$ does not increase by (\eqref{lambda-contraction}).
\end{proof}

\begin{corollary}[H1, H2]\label{corollary-minimizer}
There is a minimizer $u_{R,\mu}$ of the constrained problem that satisfies
\begin{equation}\label{urm-minimizer}
u_{R,\mu}(\Omega^1_{R,\mu}) \subseteq \overline{F}.
\end{equation}
\end{corollary}

Next, we need an {\em a priori} bound.

\begin{lemma}\label{a-priori-bound-lemma}
There is an $M>0$, independent of $R$, $\mu$, $n$, such that
\[
|u_{R,\mu}(x)| < M, \text{ for } x \in \Omega_{R,\mu}.
\]
\end{lemma}
\begin{proof}
For the convex set $C_0$ introduced in (H1), we consider the mapping $\Lambda : \R^2 \to C_0$,
\begin{equation}\label{lambda-mapping}
\Lambda u := 
\left\{\begin{array}{lll}
Pu, & \text{if} & u \notin C_0, \medskip\\
u, & \text{if} & u \in C_0,
\end{array}\right. 
\end{equation}
where $Pu$ is the projection of $u$ on $\partial C_0$. By (H1), $W(\Lambda u) \leq W(u)$. Also, the mapping $\Lambda$ is nonexpansive in the Euclidean norm. We set $(Lu)(x) := \Lambda u(x)$ and notice that $L$ preserves equivariance, honors the constraint, and reduces $J_{R,\mu}$. It follows that the minimizer $u_{R,\mu}$ of the constrained problem takes values in $C_0$. Thus \eqref{lambda-mapping} holds.
\end{proof}

\section{Local estimates}
Given $u : x \in \R^2 \to \R^2$, we write $u(x)-a^\pm$ in polar form,
\[
u(x) - a^\pm = |u(x)-a^\pm| \frac{u(x)-a^\pm}{|u(x)-a^\pm|}= \rho^\pm(x) n^\pm(x),
\]
with $\rho^\pm : x \in \R^2 \to \R_+$ and $n^\pm : x\in\R^2 \to\mathbb{S}^1$. So, if $u \in U^c_{R,\mu}$, we have
\[
u(x) = a^+ + \rho^+ (x) n^+(x), \text{ with } \rho^+(x) \leq r, \text{ for } x \in C^+_{R,\mu,\eta},
\]
and similarly for $x \in C^-_{R,\mu,\eta}$. We notice that the polar form is well defined for $\rho(x) \neq 0$.

For $u \in W^{1,2}_{\text{loc}}$, it follows that $\rho$, $n \in W^{1,2}_{\text{loc}}$ and moreover, $|\nabla
u|^2 = |\nabla\rho|^2 + \rho^2| \nabla n|^2$. On the other hand, on the set $\{u=a\}$, we have $|\nabla u|=0$ a.e. Therefore, for any measurable set $S$, we have
\[
\int_S| \nabla u|^2\, \dd x = \int_{S \cap \{ \rho>0 \}} \left\{ |\nabla\rho(x)|^2 + \rho^2(x) | \nabla n(x)|^2 \right\} \dd x.
\]

\begin{lemma}[H1]\label{lem-estimate}
Suppose $u_{R,\mu}$ is a minimizer of the constrained problem \eqref{min-problem}. Then, the following estimate holds
\begin{equation}\label{cosh-formula}
\rho^+_{R,\mu}(x) \leq r \frac{\cosh(c(R\mu - x_1))}{\cosh(c(\mu - \eta)R)}, \text{ a.e.\ } x \in C^+_{R,\mu,\eta},
\end{equation}
where $c$ as in {\em (H1)}, with an analogous estimate for $x\in C^-_{R,\mu,\eta}$. Here, $1\leq R<\infty$, $1\leq\mu\leq\infty$, and $1/2<\eta<\mu$, for $r<r_0$, $\eta$, $r_0$ fixed.
\end{lemma}

\begin{proof}
Suppose that
\begin{equation}\label{elliptic-w-problem}
\left\{\begin{array}{l}
\Delta w - c^2 w \geq 0, \medskip\\
B w \leq 0,
\end{array}\right.  
\end{equation}
weakly in the space $W^{1,2}_\# (C^+_{R,\mu,\eta})$, the latter defined as the completion in the $W^{1,2}$ norm of the space
\[
\left\{ f \in C^\infty (\overline{C^+_{R,\mu,\eta}}) \cap W^{1,2}(C^+_{R,\mu,\eta}) \mid f^+ = 0 \text{ on } \{x_1 = \eta R\} \right\},
\]
where
\[
Bw := \left\{\begin{array}{ll}
w, & \text{on } x_1=\eta R, \medskip\\
\dfrac{\partial w}{\partial n}, & \text{on } \partial_L C^+_{R,\mu,\eta}\ ( := \partial C^+_{R,\mu,\eta}\setminus \{ x_1 = \eta R \}), \\
\end{array}\right.
\]
and \eqref{elliptic-w-problem} is meant in the sense
\begin{equation}\label{inthesense}
\int_{C_{R,\mu,\eta}} \left\{ \nabla w \nabla\phi + c^2 w \phi \right\} \dd x \leq 0,
\end{equation}
for $w$, $\phi \in W^{1,2}_\# (C^+_{R,\mu,\eta})$, with $\phi \geq 0$ a.e. Then, we claim that
\begin{equation}\label{w-negative}
w \leq 0, \text{ a.e.\ in } C^+_{R,\mu,\eta}.
\end{equation}

To prove the claim, by density we can take $\phi:=w^+$ in \eqref{inthesense} and so we can conclude that
\[
0 \geq \int_{C^+_{R,\mu,\eta}} \left\{ \nabla w \nabla
w^+ + c^2w w^+ \right\} \dd x = \int_{C^+_{R,\mu,\eta}} \left\{ |\nabla w^+|^2 + c^2 |w^+|^2 \right\} \dd x = 0,
\]
thus, $w^+ = 0$ in $C^+_{R,\mu,\eta}$. Next we will show that
\begin{equation}\label{deltarho-weakly}
\Delta \rho_{R,\mu} \geq \rho_{R,\mu} c^2 \text{ weakly in } W^{1,2}(C^+_{R,\mu,\eta}).
\end{equation}
For showing \eqref{deltarho-weakly}, we consider $u_\varepsilon(x) = u_{R,\mu}(x) + \varepsilon \hat{p}(x) n(x)$, with $\hat{p}(x) \leq 0$ in $C^+_{R,\mu,\eta}$, $\hat{p}\in C^\infty_0(C^+_{R,\mu,\eta})$. We
notice that $|u_\varepsilon(x) - a^\pm| = |\rho_{R,\mu}(x)+\varepsilon\hat{p}(x)| \leq r$ in $C^\pm_{R,\mu,\eta}$. Then,
\begin{align*}
&\frac{\dd}{\dd \varepsilon} \Big|_{\varepsilon=0} J(u_\varepsilon) \geq 0 \ \Leftrightarrow \ \frac{\dd}{\dd \varepsilon}\Big|_{\varepsilon=0} \int_{\Omega_{R,\mu,\eta}} \left\{\frac{1}{2} |\nabla u_\varepsilon|^2 + W(u_\varepsilon) \right\} \dd x \geq 0 \bigskip\\
&\Leftrightarrow \int_{C_{R,\mu,\eta}} \left\{ \nabla\rho_{R,\mu} \nabla \hat{p} + \rho_{R,\mu}\hat{p} |\nabla
n(x)|^2 + \hat{p} W_u (u_{R,\mu}) n(x) \right\} \dd x \geq 0,
\end{align*}
from which it follows that
\[
\int_{C_{R,\mu}} \left\{ \nabla\rho_{R,\mu} \nabla\hat{p} + \hat{p} W_u(u_{R,\mu}) n(x) \right\} \dd x \geq 0.
\]
Utilizing (H1), we obtain
\[
\int_{C_{R,\mu}} \left\{ \nabla\rho_{R,\mu} \nabla \hat{p} + c^2 \hat{p} \rho_{R,\mu} \right\} \dd x \geq 0,
\]
and therefore \eqref{deltarho-weakly} has been established.

Next we will show that $\rho_{R,\mu} < r$ a.e.\ in the interior of $C_{R,\mu,\eta}$ from which it will follow, up to a modification on a null set, that $u_{R,\mu}$ is a classical solution of
\begin{equation}\label{interior-elliptic}
\Delta u_{R,\mu} - W_u (u_{R,\mu}) = 0, \text{ in the interior of } C_{R,\mu,\eta}.
\end{equation}
Suppose now for the sake of contradiction that $\rho_{R,\mu} = r$ on a set $A$ of positive measure. However, this is in conflict with $\Delta\rho_{R,\mu} \geq c^2\rho_{R,\mu}$ in $W^{1,2}(C_{R,\mu})$ since $\nabla\rho_{R,\mu}=0$ a.e.\ on this set $A$. Therefore, $\rho_{R,\mu}(x) < r$ a.e.\ in $C^+_{R,\mu,\eta}$ as required.

In the following we show that
\begin{equation}\label{normal-derivative-zero}
\frac{\partial\rho_{R,n}}{\partial n} = 0 \text{ on } \partial_L C_{R,\mu,\eta} \setminus \{A,B\},
\end{equation}
where $A$, $B$ are the corners. For $x^\ast$ in a subset of points $\partial_L C_{R,\mu}\setminus\{A,B\}$ such that $\rho_{R,\mu}(x^\ast)<r$ a.e.\ on it, the natural boundary conditions hold classically and so \eqref{normal-derivative-zero} is valid. Therefore, the case of interest is when $\rho_{R,\mu}(x^\ast) = r$. We notice that in the interior of $C_{R,\mu,\eta}$, \eqref{interior-elliptic} is satisfied classically and that $u_{R,\mu}$ is regular. From the bound $|u_{R,\mu}|<\text{constant}$, which holds uniformly in the
interior of $C_{R,\mu,\eta}$, we obtain by elliptic regularity that $|\nabla \rho_{R,\mu}| < \text{constant}$ on the boundary with a similar estimate on the second-order derivatives. Consequently, $\rho_{R,\mu}(x)$ is continuous at $x^\ast$ and the outer normal derivative $\partial\rho_{R,\mu}/\partial n$ exists at $x^\ast$. We know that $\Delta\rho_{R,\mu}\geq c^2\rho_{R,\mu}$ classically in the interior of $C_{R,\mu}$ and by the preceding argument, $\rho_{R,\mu}$ is continuous at $x=x^\ast$ and $\partial\rho_{R,\mu}/\partial n (x^\ast)$ exists. Applying the Hopf lemma, we obtain
\begin{equation}\label{hopf-lemma}
\frac{\partial\rho_{R,\mu}}{\partial n}(x^\ast) > 0.
\end{equation}
We now set $u_\varepsilon(x) = u_{R,\mu} + \varepsilon\hat{p}(x) n$, $\hat{p} \leq 0$ smooth
with $\text{supp}(\hat{p}) \subseteq B(x^\ast;\delta) \cap \overline{C_{R,\mu,\eta}}$, $0<\delta \ll1$. Then,
$u_\varepsilon \in U^c_{R,\mu}$ and
\begin{equation}\label{contradiction}
0 \leq \frac{\dd}{\dd \varepsilon} \Big|_{\varepsilon=0} \int_{\Omega_{R,\mu}} \left\{ \frac{1}{2} |\nabla u_\varepsilon|^2 + W(u_\varepsilon) \right\} \dd x = \int_{\partial\Omega_{R,\mu}} \frac{\partial\rho_{R,\mu}}{\partial n}\hat{p}\, \dd S,
\end{equation}
from \eqref{interior-elliptic}, which however is in contradiction to \eqref{hopf-lemma}. Therefore, $\rho_{R,\mu}(x^\ast) = r$ cannot possibly hold and so \eqref{normal-derivative-zero} is valid.

To conclude, we set
\[
v := \rho^+_{R,\mu}(x) - r \frac{\cosh(c(R\mu-x_1))}{\cosh(c(\mu-\eta)R)}.
\]
We will show that $v$ satisfies \eqref{elliptic-w-problem}. By the preceding argument, it follows that $\Delta v - c^2 v \geq 0$ classically in the interior of $C^+_{R,\mu,\eta}$. Thus, given $\phi$ as in the definition of \eqref{deltarho-weakly}, we have
\begin{align*}
0 \leq \int_{C^+_{R,\mu}} \left\{ \Delta v - c^2 v \right \} \phi\, \dd x 
&= \int_{C^+_{R,\mu}} \left\{ -\nabla v \nabla\phi - c^2  \phi \right\} \dd x + \int_{\partial_L C_{R,\mu,\eta}}\frac{\partial v}{\partial n}\phi\, \dd S \bigskip\\
&= \int_{C^+_{R,\mu}} \left\{ -\nabla v \nabla\phi - c^2 v \phi \right\} \dd x,
\end{align*}
from \eqref{normal-derivative-zero}. Finally, we note that the points $A$, $B$ are negligible in the boundary integral since $|\nabla v| < \text{constant}$, up to the boundary. The proof of lemma is complete.
\end{proof}

Taking $\mu\to\infty$ in Lemma \ref{lem-estimate}, we obtain
\begin{corollary}[H1]\label{corollary-cosh}
For $\mu=\infty$, the minimizer $u_R$ satisfies the estimate
\[
\rho^+_R(x) \leq r \mathrm{e}^{ -c(x_1-\eta R)}, \text{ for } x_1 \geq \eta R.
\]
\end{corollary}

\section{Global estimates (H1, H2, H3)}\label{global-estimates}
\begin{theorem}[H1, H2, H3]\label{global-theorem}
Suppose $r<r_0$ and $\mu=\infty$ as in the definition of the constrained problem in Section \ref{constrained-sec}. We denote the minimizer by $u_R$ and the domain by $\Omega_R$ respectively, and assume that it possesses the property in Corollary \ref{corollary-minimizer}, that is, $u_R$ is positive.

Then, there exists $R_0>0$, such that for $x\in\Omega_R$ the estimate
\begin{equation}\label{exponential-estimate}
|u_R(x) - a^+| < M \mathrm{e}^{-c |x_1|}, \text{ for } x_1 \geq 0,\, R\geq R_0,
\end{equation}
holds, where $M$ is a constant depending on the set $C_0$ in (H2).
\end{theorem}

\begin{proof}
\medskip
{\em Step 1.} We begin by noting that by Lemma \ref{a-priori-bound-lemma}, we may assume
that $u_R(x)\in C_0$.

\medskip

{\em Step 2}. Suppose $Q(u)$ a $C^2$ convex function as in (H3). We can check easily that the following holds true
\begin{equation}\label{step2-trace}
\Delta Q(u(x)) = \tr \left\{ (\partial^2 Q)(\nabla u)(\nabla u)^\top \right\} +  Q_u (u(x)) \cdot \Delta u(x) \geq Q_u(u(x)) \cdot \Delta u(x). 
\end{equation}

\medskip

{\em Step 3}. Let $u_R$ be the minimizer. Then,
\begin{equation}  \label{step3-axb}
Q(u_R(x)) \leq A x_1 + B =: U(x_1,\eta R), \text{ for } x_1 \in [0,\eta R],\, x=(x_1, x_2),
\end{equation}
where $A=(r-B)/\eta R$, $B$ a bound, and $Q(u_R(x))\leq B$, for $x\in\Omega_R$, provided by Step 1.

To prove \eqref{step3-axb}, from \eqref{step2-trace} in $\Omega_R \cap \left\{ 0 \leq x_1 \leq \eta R \right\}$, we have
\[
\Delta Q(u_R(x)) \geq Q_u(u_R(x)) \cdot W_u(u_R(x)) \geq 0
\]
by \eqref{step2-trace}, \eqref{urm-minimizer}, (H3). Then, \eqref{step3-axb} follows by the maximum principle.

\begin{figure}
\includegraphics[scale=.225]{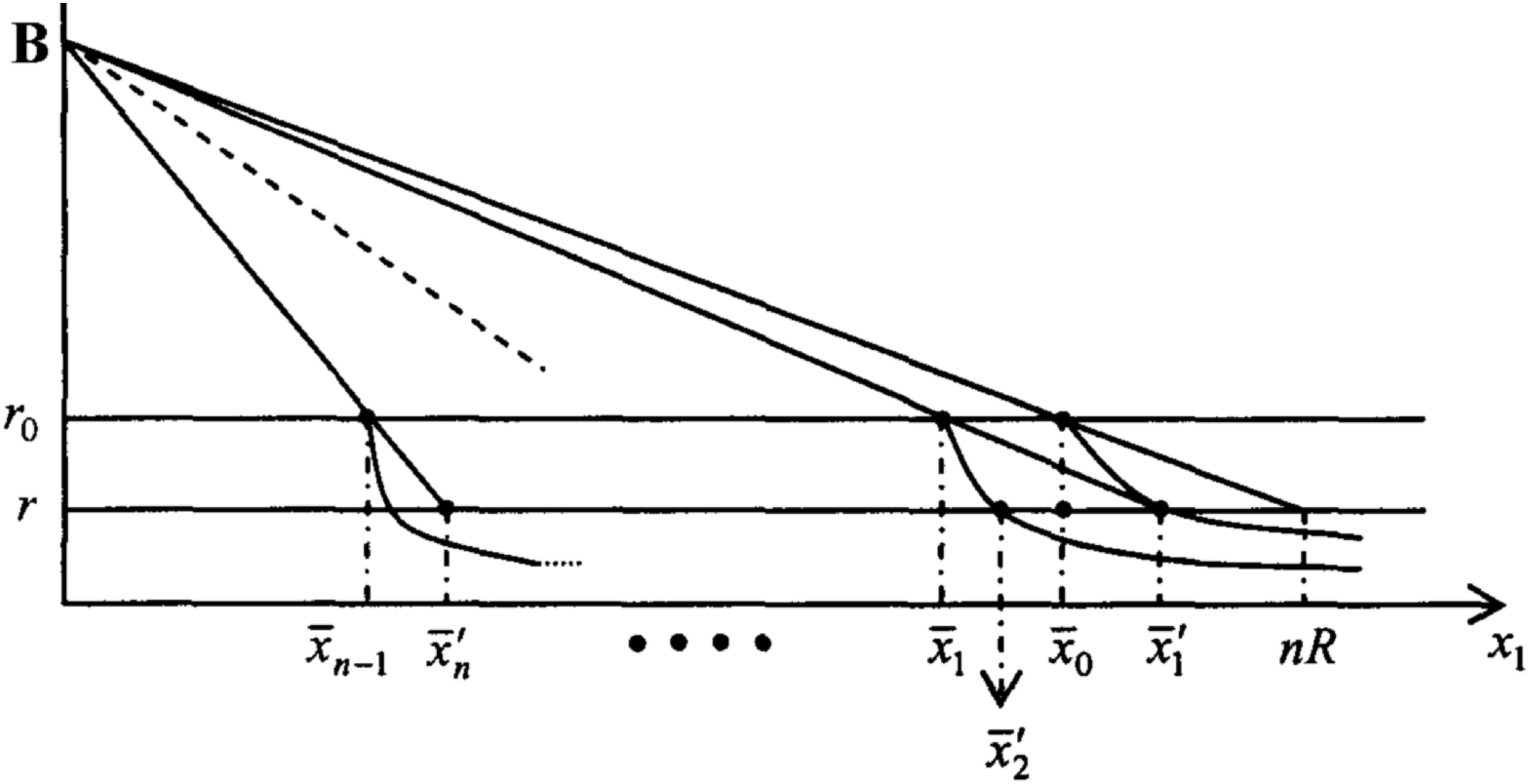}
\caption{}
\label{figure-compl}
\end{figure}

We shall denote by $U(x_1; \theta)$ the function $((r-B) / \theta) x_1 + B$. Then, $Q(u_R(x)) \leq U(x_1; \eta R)$, for $0 \leq x_1 \leq \eta R =: \bar{x}'_0$. Next, we consider the equation
\begin{equation}\label{ux-eta-r}
U (x_1; \eta R) = r_0.
\end{equation}
which has the unique solution
\[
\bar{x}_0=\frac{B - r_0}{B - r} \eta R = \delta \eta R = \delta \eta R, \text{ with } \delta = \frac{B - r_0}{B - r},\, 0<\delta<1.
\]
By the definition of $Q$, $\rho^+_R \leq r_0$, $\bar{x}_0 \leq x_1$ from which we obtain, via Lemma \ref{lem-estimate}, for $\bar{x}_0$ in the place of $\eta R$,
\begin{equation} \label{in-the-place}
\rho^+_R(x) \leq r_0 \mathrm{e}^{-c(x_1 - \bar{x}_0)} =: r_0 \sigma(x_1; \bar{x}_0), \text{ for } \bar{x}_0 \leq x_1.
\end{equation}

Now we continue the iteration. Let $\bar{x}'_1$ be the solution to $r_0 \sigma (x_1; \bar{x}_0) = r$. As before, we have $Q(u_R(x)) \leq U(x_1; \bar{x}'_1)$, for $x_1 \in [0,\bar{x}'_1]$, and therefore $\rho^+_R(x) \leq r_0$, for $x_1 \in [\bar{x}_1, \mu R]$, where $\bar{x}_1$ the solution to $U(x_1; \bar{x}'_1) = r_0$. Consequently, we have the estimate
\[
\rho^+_R(x) \leq r_0 \sigma(x_1; \bar{x}_1), \text{ for } \bar{x}_1 \leq x_1 \leq \mu R.
\]
We denote the solution to $r_0\sigma(x_1; \bar{x}_1) = r$ by $\bar{x}'_2$ and keep going, thus generating two sequences $\{\bar{x}_i\}$, $\{\bar{x}'_i\}$, for $i=1,2,\ldots$

The iteration is terminated if for some $i$, the slope of the line $U(x_1; \bar{x}'_i)$, which is $(r-B)/ \bar{x}'_i$, gets equal or less than $-cr_0$, the lower bound of the slope of $r_0 \sigma (x_1;\bar{x}_i)$ at
the point $\bar{x}_i$. Consequently, since $\bar{x}'_i$ is decreasing as $i \to \infty$ and
\[
\left| \frac{\dd}{\dd x_1} \Big|_{\bar{x}'_i} \sigma(x_1; \bar{x}_i) \right| \leq c,
\]
we may let $i \to \infty$. The iteration is terminated independently of $R$ and at a distance
\[
\lim_{i \to\infty} \bar{x}'_i = \frac{B-r}{cr_0} =: \delta^\ast
\]
from the line $x_1=0$. Moreover, we have
\[
\rho^+_R(x) \leq r_0 \sigma(x_1;\lim_{i \to\infty}\bar{x}_i) ~\text{ and }~ \lim_{i\to\infty} \bar{x}_i \leq x_1,
\]
from which it follows that $\rho^+_R(x) \leq r$, for $x_1\geq \delta^\ast$ and $x_1$. Thus,
\[
\rho^+_R(x) \leq r_0 \mathrm{e}^{-c(x_1-\eta R)}, \text{ for } \delta^\ast \leq x_1.
\]

Note that
\[
R_0 = -\frac{\ln(r / 2r_0)}{c\delta},\quad \delta=\frac{B-r_0}{B-r},\quad \delta^\ast=\frac{B-r}{cr_0}.
\]
The proof is complete. 
\end{proof}

\section{Proof of Theorem \ref{theorem} (H1, H2, H3, H4, H5)}
In this section we will work with the infinite strip, which we denote by $\Omega_R$. The constrained problem in Section \ref{constrained-sec} provides a minimizer $u_R$ which may be assumed to possess the positivity property by Corollary \ref{corollary-minimizer}. Moreover, $u_R$ satisfies the
uniform exponential bound \eqref{exponential-estimate}. By standard local estimates, the following limit exists.

\begin{equation}\label{limit-exists}
u(x) = \lim_{R_n \to \infty} u_{r_n} (x).
\end{equation}
From \eqref{exponential-estimate} we obtain
\[
|u(x) - a^+| < M \mathrm{e}^{-c|x_1|}, \text{ for } x_1 \geq 0.
\]
Parts (i) and (ii) of Theorem \ref{theorem} have been established.

\medskip

{\em Step 1.} (Upper Bound)
\begin{equation}\label{upper-bound}
J_{\Omega_R}(u) \leq C + 2R E_{\min},  
\end{equation}
where
\[
J_{\Omega_R} (v) := \int_{\Omega_R} \left\{ \frac{1}{2} |\nabla v|^2 + W(v) \right\} \dd x.
\]

First we establish
\begin{equation}\label{ur-upper-bound}
J_{\Omega_R}(u_R) \leq C + 2R E_{\min}. 
\end{equation}
For this purpose consider the comparison map
\[
\tilde{u}(x_1, x_2)=\left\{\begin{array}{ll}
e_+(x_1), &\text{for } x_2 \geq1, \bigskip\\
\left( \dfrac{1+x_2}{2} \right) e_+(x_1) + \left( \dfrac{1-x_2}{2} \right) e_-(x_1), & \text{for } |x_2| \leq 1,\bigskip\\
e_-(x_1), & \text{for } x_2\leq -1.
\end{array}\right.
\]
Note that $\tilde{u}$ is positive, equivariant, and satisfies the
pointwise constraint in Proposition \ref{constrained}. Consequently
\begin{equation}\label{consequently}
J_{\Omega_R}(u_R) \leq J_{\Omega_R}(\tilde{u}) \leq C+2R E_{\min}.
\end{equation}
Next fix $R$, choose $R' > R$, otherwise arbitrary, and notice that
\begin{equation}\label{notice-that}
J_{\Omega_R'}(u_{R'}) = J_{\Omega_R}(u_{R'}) + \int_{R<|x_2|<R'} \left\{ \frac{1}{2}|\nabla u_{R'}|^2 + W(u_{R'}) \right\} \dd x.
\end{equation}
Set
\[
V^{R'}_{x_2}(x_1) = u_{R'}(x_1,x_2),
\]
and notice that by the exponential estimate \eqref{exponential-estimate} and the variational characterization of $e_\pm$ \cite[Th.\ 3.7]{5} we have the estimate
\begin{equation}\label{have-the-estimate}
E(V^{R'}_{x_2}) \geq E(e_\pm) = E_{\min}.  
\end{equation}
Hence
\begin{align} \label{double-integral-estimate}
&\iint_{R<|x_2<R'} \left\{ \frac{1}{2}|\nabla u_{R'}|^2 + W(u_{R'}) \right\} \dd x_1 \dd x_2 \nonumber\\
&\qquad \geq \int_{R<|x_2|<R'}E(V^{R'}_{x_2})\, \dd x_2 \geq 2E_{\min}(R'-R). 
\end{align}
On the other hand
\begin{equation} \label{other-hand}
J_{\Omega'_R}(u_{R'}) \leq C+2E_{\min}R'. 
\end{equation}
Thus by \eqref{notice-that} we obtain
\[
C + 2R' E_{\min} \geq J_{\Omega_R}(u_{R'} + 2(R'-R) E_{\min},
\]
from which we obtain
\begin{equation} \label{from-which-we-obtain}
C + 2R E_{\min} \geq J_{\Omega_R}(u_{R'}), \text{ for }R'>R. 
\end{equation}
Taking $R'\to\infty$, we obtain \eqref{upper-bound}.

\medskip

{\em Step 2.} (Lower Bound)
\begin{equation} \label{lower-bound}
J_{\Omega_R}(u)\geq 2RE_{\min}. 
\end{equation}
To see this, first notice that by \eqref{have-the-estimate}
\[
\int_{|x_2|<R} E(V^{R'}_{x_2})\, \dd x_2 \geq E_{\min}(2R), \text{ for }R'>R,
\]
that is,
\begin{equation}\label{double-integral-2}
\iint_{|x_2|<R} \left\{ \frac{1}{2}|\nabla u_{R'}|^2 + W(u_{R'}) \right\} \dd x_1\dd x_2 \geq E_{\min}(2R).
\end{equation}

By utilizing the exponential estimate \eqref{exponential-estimate} and elliptic estimates (on the gradient) one can apply the dominated convergence theorem and obtain
\begin{align*}
\lim_{R'\to\infty} \iint_{|x_2|<R} &\left\{ \frac{1}{2}|\nabla u_{R'}|^2 + W(u_{R'}) \right\} \dd x_1 \dd x_2 \medskip\\
&= \iint_{|x_2|<R} \left\{ \frac{1}{2} |\nabla u|^2 + W(u) \right\} \dd x_1 \dd x_2.
\end{align*}
Thus, by \eqref{double-integral-2} we obtain \eqref{lower-bound}.

Combining Step 1.\ and Step 2.\ above we obtain part (iii) of Theorem \ref{theorem}. Notice that (H4) has not been invoked so far.

\medskip

{\em Step 3.}
\begin{equation}\label{step3-estimate}
\iint_{\R^2} \left| \frac{ \partial u}{\partial x_2} \right|^2 \dd x_1 \dd x_2 < \infty,
\end{equation}
from \cite[(5.9)]{1}.

First we establish
\begin{equation}\label{establish}
\int_{\Omega_R} \left| \frac{\partial u_R}{\partial x_2} \right|^2 \dd x < C,
\end{equation}
from which \eqref{step3-estimate} follows immediately, since for a given compact set $K\subset\Omega_R$, it follows that $\int_K |\partial u_R / \partial x_2 |^2 \dd x < C$, for $C$
independent of $R$ and $K$.

Note that
\begin{align*}
C + 2 E_{\min} E\overset{\eqref{upper-bound}}{\geq} J_{\Omega_R}(u_R) &= \int_{\Omega_R} \frac{1}{2}
\left| \frac{\partial u_R}{\partial x_2} \right|^2 \dd x + \int_{|x_2|<R} E(V^R_{x_2})\, \dd x_2 \medskip\\
&\geq\int_{\Omega_R} \frac{1}{2} \left| \frac{\partial u_R}{\partial x_2} \right|^2 \dd x + \int_{|x_2|<R} E_{\min}\, \dd x_2,
\end{align*}
and by the exponential estimate and the variational characterization of the elements of $\mathcal{M}$ in \cite{5}, the last quantity equals
\[
\int_{\Omega_R} \frac{1}{2} \left| \frac{\partial u_R}{\partial x_2} \right|^2 \dd x + 2R E_{\min},
\]
hence \eqref{establish} follows.

\medskip

{\em Step 4.} From \eqref{step3-estimate} we obtain that given any sequence $x^n_2 \to +\infty$, there is a subsequence ${x^{n}_2}'$ such that
\begin{equation}\label{subsequence-limit}
u(x_1,{x_2^{n}}') \to \theta (x_1),
\end{equation}
where
\begin{equation}\label{theta-problem}
\frac{\partial^2\theta}{\partial x^2_1} - W_u(\theta) = 0. 
\end{equation}
This is via standard elliptic estimates (see \cite[Lemma 5.2]{1}).

The exponential estimate for $u(x_1,x_2)$ implies that
\begin{equation}\label{theta-connection}
\theta(\pm\infty) = a^\pm,  
\end{equation}
that is, $\theta$ is a connection.

We will establish that the limit as $x_2\to\infty$ exists in \eqref{subsequence-limit} and that $\theta\in\mathcal{M}$, that is, a minimizing connection.

We first observe that at least along a sequence $x^n_2\to\infty$,
\[
u(\cdot, x^n_2) \to \mathcal{M}.
\]
Indeed, if not, then
\begin{equation}
\liminf_{|x_2|\to\infty} E(u(\cdot,x_2)) > E_{\min},
\end{equation}
by the finiteness of $\mathcal{C}\setminus\mathcal{M}$, but this is in conflict with the Upper Bound \eqref{upper-bound}.

Finally we will show arguing by contradiction that it is not possible to have two sequences $x^{n_1}_2$ and $x^{n_2}_2$, tending to $\infty$, such that
\begin{equation}
\left\{\begin{array}{l} \label{u-theta-ab}
u(x_1,x^{n_1}_2) \to\theta^A(x_1), \medskip\\
u(x_1,x^{n_2}_2)\to\theta^B(x_1),
\end{array}\right. 
\end{equation}
where $\theta^A \in \mathcal{M}$, $\theta^B\in\mathcal{C}$, $\theta^A \neq \theta^B$.

To continue, we need a few observations on the set of connections $\mathcal{C}$. By symmetry, we have $\theta_1(0)=0$, $\dot{\theta}_2(0)=0$. By positivity, we have $\dot{\theta}_1(0)\geq 0$. By the equipartition (first integral), we have
\begin{equation}\label{first-integral}
\frac{1}{2} |\dot{\theta}(x_1)|^2 = W(\theta(x_1)).  
\end{equation}
Evaluating \eqref{first-integral} at $x_1=0$, we see that $\dot{\theta}^2_1(0)$ is determined by $(\theta_1(0),\theta_2(0))$, and since $\dot{\theta}_1(0)\geq 0$, it is determined completely by $c = \theta_2(0)$. In conclusion, the set of relevant connections in \eqref{u-theta-ab} is an one-parameter family determined by $\theta_2(0)$. It follows that
\begin{equation} \label{one-parameter}
\theta^A_2(0) \neq \theta^B_2(0). 
\end{equation}
By hypothesis, $\mathcal{C}$ is discrete. Therefore, we can choose $c$ such that,
\[
\theta^A_2(0) < c < \theta^B_2(0),
\]
which does not correspond to any of the connections in $\mathcal{C}$. Now, from \eqref{u-theta-ab}, $u_2(0,x^{n_1}_2)\to\theta^A_2(0)$ and $u_2(0,x^{n_2}_2)\to\theta^B_2(0)$. Therefore, by continuity there exists $x^{n_3}_2 \to +\infty$ such that $u_2(0,x^{n_3}_2)=c$. By \eqref{subsequence-limit}, there is a subsequence of $\{x^{n_3}_2\}$, say $x^{n_4}_2\to +\infty$, such that $u(x_1,x^{n_4}_2)\to\theta(x_1)$. Therefore, $u_2(0,x^{n_4}_2)\to\theta_2(0)=c$, which is a contradiction. Thus, we established that \eqref{u-theta-ab} is not possible. 

In conclusion we established that
\begin{equation}\label{final}
u(x_1,x_2)\to\theta(x_1), \text{ for } \theta\in\mathcal{M}.  
\end{equation}
The proof of Theorem \ref{theorem} is complete.

\section*{Acknowledgments}
NDA was partially supported by Kapodistrias grant No.\ 70/4/5622 at the University of Athens.

\nocite{*}
\bibliographystyle{plain}

\end{document}